\documentclass{amsart} 
\usepackage{amsmath} 
\usepackage{amsfonts}
\usepackage{amsthm} 
\usepackage{amssymb}

\newtheorem{prop}{Proposition}
\newtheorem{theo}[prop]{Theorem}

\newtheorem{cor}[prop]{Corollary}

\newtheorem{conj}[prop]{Conjecture}

\theoremstyle{remark}

\begin{document}

\title[amenability and factorization property]{Characterization of amenability by a factorization property of the group von Neumann algebra}

\author{Denis Poulin}
\address{Department of Mathematics, Carleton University, Ottawa, Ontario, Canada}
\curraddr{Department of Mathematics, 1125 Colonel By Drive, Ottawa, Ontario, K1S 5B6, Canada}
\email{dpoulin@.connect.carleton.ca}
 
\thanks{The author was supported by Carleton University.}


\date{June 16, 2011 }



\begin{abstract}
 We show that the amenability of a locally compact group $G$ is equivalent to a factorization property of $VN(G)$ which is given by $ VN(G) = \langle VN(G)^*VN(G)\rangle$. This answer partially  two problems proposed by Z. Hu and M. Neufang in their article \textit{Distinguishing properties of Arens irregularity}, Proc. Amer. Math. Soc. 137 (2009), no. 5, 1753--1761
\end{abstract}

\maketitle

		It is known that the amenability of a locally compact group $G$ can be characterized in many different ways. For example $ A(G)$ has a bounded approximate 
		identity (\cite{Lep2}; \cite[Theorem 6, p. 129]{He1}), $A(G)$ factorizes weakly, i.e. $A(G)$ is the linear span of $A(G) \cdot A(G)$ \cite[Proposition 2, p. 
		138]{Lo1}, or $A(G)$ is closed in $M(A(G), A(G))$, its multiplier algebra \cite[Theorem 1]{Lo3}. It is also known that $A(G)$ is unital if and 
		only if $G$ is compact which is equivalent to $VN(G) = UCB(\hat{G})$. This last equality can be formulated for a Banach algebra $A$ by $A^* = A^*A$. 
		It seems possible to believe that if we relax the factorization property of the dual, we could characterize the existence of a BAI 
		for $A(G)$, which is equivalent to the amenability of $G$ by Leptin's theorem. More precisely, in this note, we prove that $G$ is amenable if 
		and only if the linear span $ VN(G)^*VN(G)$, denoted by $\langle VN(G)^*VN(G)\rangle$, is equal to $VN(G)$. With this characterization of amenability for 
		locally compact groups, we give a partial answer to the following problems proposed by Z. Hu and M. Neufang in \cite{HN1}. \\
			~~\\
			\textbf{Problem 1} : Does $\overline{\langle VN(G)^*\square VN(G)\rangle}=VN(G)$ imply that $G$ is amenable?\\
			~~\\ 
		In this article, we made a stronger assumption than only the norm density of $\langle VN(G)^*\square VN(G)\rangle$ in $VN(G)$. However, with our assumption, 
		we obtain a positive answer to this problem. \\ \indent Another problem proposed in the same 
		article of Z. Hu and M. Neufang is similar to the previous one. We stated here.\\
			~~\\
			\textbf{Problem 2} : If $VN(G)$ factors over $B_{\rho}(G)$, i.e., $\overline{\langle B_{\rho}(G)\square VN(G)\rangle} = VN(G)$, is $G$ amenable?\\
		~~\\
		By \cite[Corollary 3.1]{H1}, $\overline{\langle B_{\rho}(G)\square VN(G)\rangle} = \overline{\langle VN(G)^*\square VN(G)\rangle}$ if $G$ is discrete. Thus, with Theorem 
		\ref{theo_amenable}, if $G$ is discrete such that $ \langle B_{\rho}(G)VN(G)\rangle = \langle VN(G)^*\square VN(G)\rangle$ and closed, then $\langle B_{\rho}(G)\square VN(G)\rangle = 
		VN(G)$ implies that $G$ is amenable.\\ \indent
		Our argument will use the theory of abstract Segal algebras. The reader is referred to \cite{Bu2}, \cite{Bu3}, \cite{Le2} and \cite{Le3} to know more about 
		this subject. For the definition and properties of the Fourier algebra $A(G)$ and its dual $ V(G)$, the group von Neumann algebra,  we refer to \cite{Ey1}. 
		Also, to fix our notation, we recall the definitions of the Arens products. Let $A$ be a Banach algebra. For $ m, n \in A^{**}$, $f \in A^*$, $a,b \in A$, 
												\begin{eqnarray*}
																\langle m \square n , f \rangle &=& \langle m , n\square f \rangle,
												\end{eqnarray*}
		where $	\langle n\square f ,a \rangle 	= \langle n , f \square a \rangle$ and $ \langle f \square a, b \rangle 	= \langle f , ab \rangle$. Similarly, using 
		the right $A$-module structure of $A$ on $A^*$, we have that
												\begin{eqnarray*}
															 \langle m \triangle n , f \rangle &=& \langle n , f\triangle m \rangle, 
												\end{eqnarray*}
		where $\langle f \triangle m , a \rangle = \langle m, a\triangle f \rangle$ and $\langle a \triangle f , b \rangle = \langle f , ba \rangle$. For a Banach 
		algebra $A$ and a $A$-bimodule $X$, we denote by $\langle XA\rangle $ and $\langle AX\rangle$ the linear span of $XA$ and $AX$ respectively.\\
		 ~~\\
		We present the main tool of this note which is an improvement of \cite[Theorem 3.3.3]{Po}. 
		\begin{theo}\label{theo_segal}
					Let $A$ be a faithful Banach algebra. Let $B$ be a proper right (left) abstract Segal algebra in $A$. Then $B^* \neq
					\langle B^* \triangle B^{**}\rangle$ ($ B^* \neq \langle B^{**} \square B^*\rangle$ ).
		\end{theo}
		
		\begin{proof}
					We will proceed by contradiction. Let $ B^* = \langle B^* \triangle B^{**}\rangle$. We 
					first prove that each element of $ \langle B^* \triangle B^{**}\rangle$ as a unique extension on $A$. From there, we will conclude that the norm of $A$ and $B$ are 
					equivalent on $B$ which is a contradiction. Let $ m_i \in B^{**}$, $ i=1, \ldots, n$. By Goldstein's theorem, for each $i$, there exists a bounded net $ 
					m_{i,\alpha}$ in $B$ such that $ m_{i,\alpha}\stackrel{w^*}{\rightarrow} m_i$ in $B^{**}$. Let $ f_i\in B^*$, and $ b\in B$. Observe that
								\begin{eqnarray*}
											\left| \langle \sum_{i=1}^n f_i \triangle m_i, b \rangle \right| 	&=& \left| \sum_{i=1}^n \langle m_i\triangle b , f_i \rangle \right| \\
																														&=& \left| \sum_{i=1}^n \langle m_i \square b , f_i \rangle \right| \\
																														&=& \left| \sum_{i=1}^n \lim_{\alpha}\langle m_{i,\alpha} b, f_i \rangle \right| \\
																														&\leq & \sum_{i=1}^{n} C_i \sup_{\alpha} \|m_{i,\alpha}\|_{B} ~\|b\|_{A} ~\|f_i\|_{B^*}\\
																														&\leq & M \|b\|_A
								\end{eqnarray*}
		This shows that $ \sum_{i=1}^n f_i \triangle m_i$ is bounded on $B$ with the norm of $A$. Since $B$ is dense in $A$, then there exists a unique extension of 
		$\sum_{i=1}^n f_i \triangle m_i$ on $A$. Let $\iota : B \rightarrow A$ be the natural inclusion map. Then $\iota^*$ is injective by the density of $B$, and for $ h\in 
		A^*$, $\iota^*(h) = h|_{B}$. Let us prove that $\iota^*$ is surjective. Let $ f \in B^*$. Since $B^* = \langle B^*B^{**}\rangle$, then there are $ f_i \in B^*$ 
		and $ m_i \in B^{**}$ such that $ f = \sum_{i=1}^n f_i \triangle m_i$. Now, if we apply $\iota^*$ on the extension of $f$, we get $f$. From this, we deduce 
		that $\iota^*$ and $\iota^{**}$ are isomorphisms. Hence, there exist non-zero constants $D_1$ and $D_2$ such 
		that 
					\begin{eqnarray*}
								D_1~\|m\|_{B^{**}} \leq \|\iota^{**}(m)\|_{A^{**}} \leq D_2~ \|m\|_{B^{**}}
					\end{eqnarray*}
		for all $m\in B^{**}$. Since, $\iota^{**}(b) = \iota(b) = b$ for all $b \in B$, we get from the previous inequalities that the norm of $A$ and the norm of $B$ are 
		equivalent on $B$, which contradicts the fact that $B$ is a proper abstract Segal algebra.

\end{proof}

\begin{cor}\label{cor1}
			Let $A$ be a right, respectively left, faithful Banach algebra. If $A^* = \langle A^*\triangle A^{**}\rangle$, respectively $A^* = \langle A^{**} \square A^*\rangle$, then the norm of $A$ is equivalent to its 
			right, respectively left, multiplier algebra.
\end{cor}

\begin{proof}
			We do a proof by contraposition. Suppose that the norm of $A$ is not norm equivalent to the norm of its right multiplier algebra $RM(A)$. Then $A$ is a 
			right abstract Segal algebra in its closure in $RM(A)$. By Theorem \ref{theo_segal}, $ A^* \neq \langle A^* \triangle A^{**}\rangle$.
\end{proof}

We mention here that the equivalence of the norm of a Banach algebra $A$ with its right or left multiplier algebra does not imply that $A$ has a bounded approximate identity (see \cite[Example 5]{W1}).\\ \indent
Using Theorem \ref{theo_segal}, we give a positive answer of problem 1 of \cite{HN1} if \\
 $\langle VN(G)^*\square VN(G)\rangle$ is already closed and not only norm dense. Note that the commutativity of $A(G)$ implies that $VN(G) \triangle VN(G)^* =  VN(G)^* \square VN(G)$. 

			\begin{theo}\label{theo_amenable}
						Let $G$ be a locally compact group. Then $G$ is amenable if and only if $VN(G) = \langle VN(G)^*\square VN(G)\rangle$.
			\end{theo}
			
			\begin{proof}
						If $G$ is amenable, then by Leptin's theorem, $A(G)$ has a bounded approximate identity and so $ VN(G) = VN(G)^*\square VN(G)$. If $VN(G) = \langle VN(G)^* \square VN(G)\rangle$, 
						then by Corollary \ref{cor1}, the norm of $A(G)$ and its multiplier algebra are equivalents. Thus, by \cite[Theorem 1]{Lo3}, $G$ is 
						amenable.	
			\end{proof}

	It is interested to compare this result on factorization of $VN(G)$ with the classical factorization property treated by A. Lau and A. \"{U}lger in \cite{LU1}, 
	i.e., $A^*=A^*A$ for a Banach algebra $A$ with a BAI. In the case of the Fourier algebra, for any locally compact  group $G$, the equality $VN(G) = 
	VN(G)\triangle A(G) = UCB(\hat{G})$ implies that $A(G)$ is unital. This is not the case in general for Banach algebra. Take for example, 
	$A=K(X)$, where $X$ is a non-reflexive Banach space such that $X^*$ has the bounded approximation property and the Radon-Nikodym property. By \cite[Corollary 4.2.12]{Po}, $A^*=A^*A$, but $K(X)$ is not unital. However, Lau-\"{U}lger asked a question after \cite[Theorem 2.6]{LU1} which was made as a conjecture by the author \cite[Conjecture 5.3.5]{Po}, we state it here:
			\begin{conj}
						There is no faithful infinite-dimensional non-unital weakly sequentially complete Banach algebra $A$ with a BAI such that $A^* = A^*A$. 
			\end{conj}
	This conjecture is true for many Banach algebras like strongly Arens irregular Banach algebras, Arens regular Banach algebras, Banach algebras which are ideals 
	in their biduals. To get a more complete list with proof of each case, see \cite[Theorem 5.3.6]{Po}. Theorem 
	\ref{theo_amenable} would suggest that for weakly sequentially complete Banach algebra, the factorization $A^*=\langle A^* \triangle A^{**}\rangle$ captures the 
	existence of a BAI. We can reformulate this by saying that the factorization $A^*=\langle A^* \triangle A^{**}\rangle$ can characterize the co-amenability of a locally compact quantum group $\mathcal{G}$.

\end{document}